\documentclass{amsart}

\usepackage{hyperref}

\usepackage{amssymb}
\usepackage{amsmath}
\usepackage{amsthm}
\usepackage{enumerate}
\usepackage{graphicx}
\usepackage{color}

\theoremstyle{plain}
\newtheorem{thm}{Theorem}[section]
\newtheorem{lem}[thm]{Lemma}
\newtheorem{prop}[thm]{Proposition}
\newtheorem{cor}[thm]{Corollary}
\newtheorem{ex}[thm]{Example}
\newtheorem{de}[thm]{Definition}

\newcommand{\diag}{{\operatorname{diag}}}
\newcommand{\lcm}{{\operatorname{lcm}}}

\begin{document}
	
	\title[Limiting weak type behaviors of maximal operator]
	{Limiting weak type behaviors of maximal operator on the positive real axis}
	
	\author{Wu-yi Pan and Sheng-jian Li}

	\address{$^1$ Key Laboratory of Computing and Stochastic Mathematics (Ministry of Education),
		School of Mathematics and Statistics, Hunan Normal University, Changsha, Hunan 410081, P.
		R. China}
	
	\email{pwyyyds@163.com (W. Pan)}
	
	\email{carlst@163.com (S. Li)}
	

	\date{\today}
	
	
	\thanks{The research is supported in part by the NNSF of China (Nos. 11831007 and 12071125)}

	
	\begin{abstract}
		In this note, we establish a discrete method to characterize the limiting weak type behaviors of the centered Hardy-Littlewood maximal operator on the positive real axis through testing on Dirac deltas. As an application, we give some new examples of the limiting weak-type behaviors of the maximal operator associated with several special measures.
		
	\end{abstract}
	
	\maketitle
	
	\section{ Introduction }

	Let $(X,\rho, \mu)$ be a metric measure space, where $\rho$ is a metric and $\mu$ is a Borel measure on metric space $(X,\rho)$. In this paper, we always assume that the measure $\mu$ is finite on each open ball and  the support of $\mu$ is nonempty, and denote a metrically closed ball by $B(x,r)=\{y\in X:\rho(x,y)\leq r\}$. We consider the centered maximal operator $M_{\mu}$ acting on a finite  measure $\nu$ by
	\begin{equation*}
		M_{\mu}\nu(x)=  \sup_{r>r_0(x)}\frac{\nu(B(x,r))}{\mu(B(x,r))},
	\end{equation*}
	here $r_0(x)=\inf\{r>0:\mu(B(x,r))>0\}$. If $\nu$ is given by $\nu(A)=\int_{A}|g|d\mu$ for some integrable functions $g\in L^1(\mu)$ and all Borel subsets $A$ of the $(X,\rho)$, then the maximal operator above reduces to the classical  centered Hardy-Littlewood maximal operator:
	\begin{equation*}
		M_{\mu}g(x):=M_{\mu}\nu(x)=\sup_{r>r_0(x)}\frac{1}{\mu(B(x,r))}\int_{B(x,r)}|g|d\mu.
	\end{equation*}
	Which is one of the two fundamental operators in harmonic analysis, and has a long research history. One of the interesting questions is  to determine what kind of measure on the metric space $(X,\rho)$ satisfies that the maximum operator on it is of weak type $(1,1)$. In the specific case of $({\Bbb R}^n, d, \mu),n\geq 1$, with the metric $d$ generated by $\Vert\cdot\Vert_{2}$ or $\Vert\cdot\Vert_{\infty}$ and an arbitrary locally finite Borel measure $\mu$, it is well known that the centered maximal operator $M_{\mu}$ is of weak type $(1,1)$, with bounds relating to $d$ but independent of the measure, since the spaces have Besicovitch Covering Property (BCP). More historical details are available in \cite{Fe69,He01,Ma95}. A new advance has been made by  Aldaz \cite{Al21},  who proved that $M_\mu$ satisfies weak type $(1,1)$ bounds that are uniform in $\tau$-additive measure $\mu$, if and only if metric measure space has the Besicovitch intersection property (WBCP),  less restrictive than BCP.
	
A more natural question is how about the limiting weak type behavior of maximal operator? It is first systematically answered by Janakiraman \cite{Ja06}, who had proved the excellent result in case of Lebesgue measure: for every finite signed measure $\nu$ on Euclidean space $\mathbb{R}^n$,
	\begin{equation*}
		\lim\limits_{\lambda\to 0^+}\lambda m\{x\in \mathbb{R}^n: M_{m}\nu(x)>\lambda\} = |\nu|(\mathbb{R}^n),
	\end{equation*}
	where we denote by $m$ the Lebesgue measure. Since then, several authors \cite{HH08,HW19} made substantial contribution in this direction. Notably, given a real number $\beta >0$, for power weighed measure $d\mu(x) = |x|^{\beta}dm(x)$ on Euclidean space $\mathbb{R}^n$, Hou and Wu \cite{HW19} showed that
	\begin{equation*}
		\lim\limits_{\lambda\to 0^+}\lambda \mu\{x\in \mathbb{R}^n: M_{\mu}\nu(x)>\lambda\} = \frac{\mu(B(0,1))}{\mu(B(e_1,1))}\nu(\mathbb{R}^n)
	\end{equation*}
	for all absolutely continuous measure $\nu$ with respect to measure $\mu$, among $e_1=(1,0,...,0)$ is a unit vector.  If we use the symbol $\delta_y$ to denote Dirac measure at the point $y$, it should be noted that the coefficient of the right hand side $\frac{\mu(B(0,1))}{\mu(B(e_1,1))}$ equals to
	\begin{equation*}
		\lim\limits_{\lambda\to 0^+}\lambda \mu\{x\in \mathbb{R}^n: M_{\mu}\delta_{0}(x)>\lambda\}.
	\end{equation*}
	Motivated by this, we can continue pose a question: 
	
	\indent {\bf Question:} For a constant $c_0 \geq 0$ and the measure $\mu$ satisfying $\lim\limits_{\lambda\to 0^+}\lambda \mu\{x: M_{\mu}\delta_{y}(x)>\lambda\}=c_0$ with a $y\in (X,\rho)$, does
	$$
	\lim\limits_{\lambda\to 0^+}\lambda \mu\{x: M_{\mu}\nu(x)>\lambda\}=c_0\nu(X)
	$$
	necessarily hold for all finite Radon measure $\nu$?

	In this note, a counterexample is given for the above question. Instead, here we prove the following result that allow us to determine the best constant in the limiting weak type $(1, 1)$ inequality for maximal convolution operators by Dirac deltas. Its complete form is the following theorem.

	\begin{thm}\label{thm:1.1}
		Assume that $\mu$ is a measure restricting on $[0, \infty)$. And let $c_0 \geq0$ be a constant. Then the following are equivalent{\rm:}
		\begin{enumerate}
			[\rm(i)]
			\item 	
			$\lim\limits_{\lambda \to 0^+}\lambda \mu(\{x\geq 0: M_{\mu}\nu(x)>\lambda\}) = c_0\nu(X)$ for all finite Radon measure $\nu${\rm;}
			\item 	$
			\lim\limits_{\lambda \to 0^+}\lambda \mu(\{x\geq 0: M_{\mu}\delta_{y}(x)>\lambda\}) = c_0$ for all $y \in [0,\infty)$.
		\end{enumerate}
	\end{thm}
	
	The paper is organized as below. In Section \ref{Se:2}, we prove  Theorem \ref{thm:1.1}.
	Next in section \ref{Se:3}, we give some new examples of the limiting weak-type behaviors of the maximal operator associated with several special measures and we finish the section by giving a counterexample of the problem mentioned above, which is in contrast to the limiting weak type behavior of maximal operator of the known measures.
	
	\section{The proof of Theorem \ref{thm:1.1}}\label{Se:2}
	 For the sake of simplicity, we	let $\Delta_{\mu}(k,\lambda):=\lambda\mu(\{x\in X:\frac{1}{\mu(B(x,k|x|))}>\lambda\})$.
	\begin{thm}\label{thm:2.1}
	 If $M$ is of the weak type $(1,1)$ and
		\[
		\lim_{k\to 1^{-}}\limsup_{\lambda \to 0^{+}}\Delta_{\mu}(k,\lambda)=\lim_{k\to 1^{+}}\liminf_{\lambda \to 0^{+}}\Delta_{\mu}(k,\lambda)=c_0,
		\]
		then
		$
		\lim_{\lambda \to 0^+}\lambda\mu(\{x\in X: M_{\mu}\nu(x)>\lambda\})=c_{0}\nu(X)
		$ for every finite  measure $\nu$.
		\end{thm}

	\begin{proof}
		Using the $1$-homogeneity of maximal operator $M$, we can assume that $\nu(X)=1.$ Then for any $0<\varepsilon <1$, there exists a ~$r_{\varepsilon}>0$~ such that
		$
		\nu(B(0,r_{\varepsilon}))>1-\varepsilon.
		$ Set $R_{\varepsilon} =r_{\varepsilon}(1+\frac{1}{\varepsilon}).$	For ~$\lambda>0$, to estimate the total mass of $E_{\lambda} := \{ M\nu(x)>\lambda\}$, we consider to measure the head $E^{1,1}_{(1-\sqrt{\varepsilon})\lambda}:= \{ M\nu_{1}>(1-\sqrt{\varepsilon})\lambda\}\cap B(0,R_\varepsilon)$, the body $E^{1,2}_{(1-\sqrt{\varepsilon})\lambda}:= \{ M\nu_{1}>(1-\sqrt{\varepsilon})\lambda\}\cap B^c(0,R_\varepsilon)$ and the tail $E^{2}_{\sqrt{\varepsilon}} = \{ M\nu_{2}>\sqrt{\varepsilon}\},$ where we denote
	$\nu_{1}=\nu|_{B(0,r_{\varepsilon})}$ and $\nu_{2}=\nu|_{B^c(0,r_{\varepsilon})}$. Obviously, the weak type property of $M$ controls the size of $E^{2}_{\sqrt{\varepsilon}}$, so it is sufficient to handle the body.
		If $|x|>R_{\varepsilon}$, we get
		\[
		M_{\mu}\nu_{1}(x)= \sup_{r>0} \frac{\nu_{1}(B(x,r))}{\mu(B(x,r))} \leq
		\frac{1}{\mu(B(x, |x|- r_{\varepsilon}))}.
		\]
On the other hand, we have $|x|-r_{\varepsilon}>|x|/(1+\varepsilon)$ thus
\begin{equation}\label{eq:2}
\begin{aligned}
\mu(E^{1,1}_{(1-\sqrt{\varepsilon})\lambda})\leq	&		 \mu(\{|x|> R_{\varepsilon}:\frac{1}{\mu(B(x,|x|-r_{\varepsilon}))}>(1-\sqrt{\varepsilon})\lambda\})\\
			\leq	&	 \mu(\{|x|> R_{\varepsilon}:\frac{1}{\mu(B(x,\frac{|x|}{1+\varepsilon}))}>(1-\sqrt{\varepsilon})\lambda\}).
\end{aligned}
\end{equation}
Letting the weak $(1,1)$ constant be $C$, we deduce
		\[
		\begin{aligned}
			\mu(E_{\lambda})\leq& \mu(E^{1,1}_{(1-\sqrt{\varepsilon})\lambda}) +\mu(E^{1,2}_{(1-\sqrt{\varepsilon})\lambda})+\frac{C\sqrt{\varepsilon}}{\lambda}\\
			\leq& \mu(E^{1,1}_{(1-\sqrt{\varepsilon})\lambda}) +\mu(\{|x|\leq R_{\varepsilon}\})+\frac{C\sqrt{\varepsilon}}{\lambda}.
		\end{aligned}
		\]
Now multiplying both sides of \eqref{eq:2} by $\lambda$ and letting $\lambda \to 0^{+}$, we have	
		\[
		\begin{aligned}
			\limsup_{\lambda \to 0^{+}}\lambda\mu(E_{\lambda})\leq &  \limsup_{\lambda \to 0^{+}}\lambda\mu(\{x:\frac{1}{\mu(B(x,\frac{|x|}{1+\varepsilon}))}>(1-\sqrt{\varepsilon})\lambda\})+ C\sqrt{\varepsilon}\\
			\leq& \frac{1}{(1-\sqrt{\varepsilon})} \limsup_{\lambda \to 0^{+}}\Delta_{\mu}(\frac{1}{1+\varepsilon},\lambda)+ C\sqrt{\varepsilon}.
		\end{aligned}
		\]
		Next letting $\varepsilon \to 0$,  we get $\limsup_{\lambda \to 0^{+}}\lambda\mu(E_{\lambda})\leq c_0$ as desired.
		
		We now estimate the lower bound.
		In view of \[ M_{\mu}\nu_{1}(x)= \sup_{r>0} \frac{\nu_{1}(B(x,r))}{\mu(B(x,r))} \geq \frac{1-\varepsilon}{\mu(B(x,|x|+r_{\epsilon}))}\] and ~$|x|+r_{\varepsilon}\leq \frac{2\varepsilon+1}{1+\varepsilon},$~ for ~$|x|>R_{\varepsilon},$ we will get
		\[
		\begin{aligned}
			\mu(E^{1}_{(1+\sqrt{\varepsilon})\lambda})\geq	&		 \mu(\{|x|> R_{\varepsilon}:\frac{1-\varepsilon}{\mu(B(x,|x|+r_{\varepsilon}))}>(1+\sqrt{\varepsilon})\lambda\})\\
			\geq &\mu(\{x:\frac{1-\varepsilon}{\mu(B(x,\frac{(2\varepsilon+1)|x|}{\varepsilon+1}))}>(1+\sqrt{\varepsilon})\lambda\})-\mu(\{|x|\leq R_{\varepsilon}\}).
		\end{aligned}
		\]
In the same manner we can see that
		\[
		\begin{aligned}
			\liminf_{\lambda \to 0^{+}}\lambda\mu(E_{\lambda})\geq &  \liminf_{\lambda \to 0^{+}}\lambda\mu(\{x:\frac{1}{\mu(B(x,\frac{(1+2\varepsilon)|x|}{1+\varepsilon}))}>(1+\sqrt{\varepsilon})\lambda\})- C\sqrt{\varepsilon}\\
			\geq& \frac{1}{(1+\sqrt{\varepsilon})} \liminf_{\lambda \to 0^{+}}\Delta_{\mu}(\frac{1+2\varepsilon}{1+\varepsilon},\lambda)- C\sqrt{\varepsilon}.
		\end{aligned}
		\]
Then $\liminf_{\lambda \to 0^{+}}\lambda\mu(E_{\lambda})\geq c_0$ and hence the desired equality follows combining the preceding result.
	\end{proof}

	The first fact is that $\lim_{k\to 1^{-}}\limsup_{\lambda \to 0^{+}}\Delta_{\mu}(k,\lambda)\leq \limsup_{\lambda \to 0^{+}}\Delta_{\mu}(1,\lambda)\leq \Vert M\Vert_{L^1\mapsto L^{1,\infty}}$ since $\Delta_{\mu}(1,\lambda)=\lambda \mu\{x: M_{\mu}\delta_{0}(x)>\lambda\}$.
 In \cite{St15}, Stempak proved that the modified maximal operators defined by 
	\begin{equation*}
		M_{k,\mu}g(x):=\sup_{r>r_0(x)}\frac{1}{\mu(B(x,kr))}\int_{B(x,r)}|g|d\mu
	\end{equation*} are of weak type (1, 1) with
	the weak type constants equaling to one, which implies $\Delta_{\mu}(k,\lambda)\leq 1$ for $k\geq 2$ in general. 
	
	If $\mu$ is a doubling measure, we record the following simple consequence: Let $C_\mu$ be the doubling constant of the measure $\mu$. Then, for all pairs of radii $0\leq r\leq R$ the inequality 
	\begin{equation*}
		\frac{\mu(B(x,R))}{\mu(B(x,r))}\leq C_\mu\left(\frac{R}{r}\right)^{\log_2 C_{\mu}}
	\end{equation*} holds true for all $x\in X.$ A basic argument shows that the  repeated limits appearing in the condition of Theorem \ref{thm:2.1} are all finite. Meanwhile, $\lim_{k\to 1^{-}}\limsup_{\lambda \to 0^{+}}\Delta_{\mu}(k,\lambda)$ can actually go to infinity for some non-doubling measures, which means that our conditions are relatively demanding. 
	
It is rather straightforward to see that the following theorem remains valid by modifying the proof of Theorem \ref{thm:2.1}.  

\begin{thm}\label{thm:2.2}
		If ~$\Vert M\Vert_{L^1\mapsto L^{1,\infty}}< \infty$,  and 
	\[
	\lim_{\lambda \to 0^{+}}\lambda\mu(\{|x|>r:\frac{1}{\mu B(x,|x|-r)}>\lambda\})=\lim_{\lambda \to 0^{+}}\lambda\mu(\{x:\frac{1}{\mu B(x,|x|+r)}>\lambda\})=c_0
	\]
	for all $r>0$, then		
$
	\lim\limits_{\lambda \to 0^+}\lambda\mu(\{x: M_{\mu}\nu(x)>\lambda\})=c_{0}\nu(X)
$ for all finite measure $\nu$.
	\end{thm}

 Let $f : [0,\infty) \to [0,\infty)$ be a non-increasing function
and let $m^n$ be Lebesgue measure on  Euclid space $\mathbb{R}^n$. Then the function $f$ defines a rotationally
invariant (or radial) measure $\mu$ via
\begin{equation}\label{eq:3}
	\mu(A)=\int_A f(|y|)dm^n(y).
\end{equation}

\begin{prop}
 Fix $n\in \mathbb{N}\backslash0$. Let
	$\mu$ be the radial measure defined via \eqref{eq:3}. If
	\begin{equation*}
		\lim\limits_{s\to\infty}\frac{\mu B(0,s)}{\mu B(se_1, s-r)}=\lim\limits_{s\to\infty}\frac{\mu B(0,s)}{\mu B(se_1, s+r)}=c_0
	\end{equation*}
for all $r>0$, then  $
\lim\limits_{\lambda \to 0}\lambda\mu(\{x: M_{\mu}\nu(x)>\lambda\})=c_{0}\nu(X)
$  for all finite  measure $\nu$.

\end{prop}
\begin{proof}
 So upon the suppositions, for all radial measure $\mu$ we have  $\mu B(x,|x|-r)=\mu B(|x|e_1, |x|-r)$ for $0<r<|x|$. Since this number depends only on $|x|$ and $r$, we can set $\phi_{r}(|x|):=\mu B(x,|x|-r).$ Hence
$\mu(\{|x|>r:\frac{1}{\mu B(x,|x|-r)}>\lambda\})=\mu B(0,\phi_{r}^{-1}(1/ \lambda))$. 	Now we use variable substitution $s=\phi_{r}^{-1}(1/ \lambda)$ to get
	\begin{equation*}
			\lim_{\lambda \to 0^{+}}\lambda\mu(\{|x|>r:\frac{1}{\mu B(x,|x|-r)}>\lambda\})=\lim\limits_{s\to\infty}\frac{\mu B(0,s)}{\mu B(se_1, s-r)}.
	\end{equation*}
The same arguments show that
\begin{equation*}
	\lim_{\lambda \to 0^{+}}\lambda\mu(\{x:\frac{1}{\mu B(x,|x|+r)}>\lambda\})=\lim\limits_{s\to\infty}\frac{\mu B(0,s)}{\mu B(se_1, s+r)}.
\end{equation*}	Thus after Theorem \ref{thm:2.2}, we prove the statement.
\end{proof}
	In particular, in the case of one dimension (the positive real axis), since every measure is radical, we get the following proposition.
	
		\begin{prop}\label{pr:2.4}
		Let m be one-dimensional Lebesgue measure and $\mu$ be a locally finite measure such that $d\mu=h\cdot \chi_{[0,\infty)}dm$ where $h$ is non-negative measurable function on $\mathbb{R}$ and $\chi_{[0,\infty)}$ is an indicator function of $[0,\infty)$. Denote  distribution function  $H(x)=\mu([0,x])$. If $H(\infty)=\infty$, for a given constant $c_0 \geq0$, then the following are equivalent{\rm :}
		\begin{enumerate}
			[\rm(i)]
			\item 	
			$\lim\limits_{r\to \infty}\frac{H(r)}{H(2r-y)}=c_0$ for all fixed $y\geq 0;$
			\item 		
			$\lim\limits_{\lambda \to 0^+}\lambda \mu(\{x\geq 0: M_{\mu}\nu(x)>\lambda\}) = c_0\nu(X)$ for every all finite measure $\nu$.
		\end{enumerate}
	\end{prop}

Now we can use Proposition \ref{pr:2.4} to prove Theorem \ref{thm:1.1}.

\begin{proof}[Proof of Theorem \ref{thm:1.1}.]
		First we note that $M_{\mu}\delta_{y}(x)=\frac{1}{\mu(B(x,|x-y|))}$ for  all Dirac measure $\delta_{y}, y\geq 0$.  Since  $\mu(B(x,\vert x-y\vert))=\int_{y}^{2x-y}h(z)dz$ for $x>y\geq 0$, we deduce
	\begin{align*}
		&    		\lambda \mu(\{x>y: M_{\mu}\delta_{y}(x)>\lambda\})
		=\lambda \mu(\{x>y:\int_{y}^{2x-y}h(z)dz<\frac{1}{\lambda}\}).
	\end{align*}
	From the assumption that $H(\infty)=\infty$, we obtain
	\begin{align*}
		\lim_{\lambda \to 0^+}\lambda\mu(\{x>y:	M_{\mu}\delta_{y}(x)>\lambda\})
		= \lim_{\lambda \to 0^+}\lambda\int_{y}^{\frac{y}{2}+\frac{1}{2}H^{-1}(\frac{1}{\lambda}+H(y))}h(z)dz.
	\end{align*}
	Now we use variable substitution $r= \frac{y}{2}+\frac{1}{2}H^{-1}(\frac{1}{\lambda}+H(y))$ to get
	\begin{align*}
		\lim_{\lambda \to 0^+}\lambda\mu(\{x>y:	M_{\mu}\delta_{y}(x)>\lambda\})
		= \lim_{r\to +\infty }\frac{H(r)-H(y)}{H(2r-y)-H(y)}.
	\end{align*}
	Finally from the locally finiteness of $\mu$, we have
	\begin{align*}
		\lim_{\lambda \to 0^+}\lambda\mu(\{x\geq 0: M_{\mu}\delta_{y}(x)>\lambda\})
		=&\lim_{\lambda \to 0^+}\lambda\mu(\{x> y: M_{\mu}\delta_{y}(x)>\lambda\})\\
		=&\lim_{r\to +\infty }\frac{H(r)}{H(2r-y)}
	\end{align*}for all fixed $y\geq 0$.
	According to Proposition \ref{pr:2.4}, the proof is complete.
\end{proof}
	
	\section{An counterexample showing failure of problem}\label{Se:3}
In this section, we will partly answer the question raised at the beginning in the exceptional circumstance of the positive real axis. In general, the answer is false.
	 
	  Throughout this section, the $(\mathbb{R},d)$
	will stand for one-dimensional Euclidean space.
	To state our result, we first give the proposition. If constant $c_0>0$, this proposition allows us to answer what kind of measure satisfies the limiting weak type behavior  of maximal operator  over $\delta_{0}$ must deduce that this measure satisfies the limiting weak type behavior of maximal operator over the class of finite measures. Results are summarized as follows.
	\begin{prop}\label{co:R}
		Let $\mu$ be a measure as assumed above. Suppose constant $c_0 > 0$, then the following statements are equivalent{\rm :}
		\begin{enumerate}
			[\rm(i)]
			\item 	
			$\lim\limits_{\lambda \to 0^+}\lambda \mu(\{x\geq 0: M_{\mu}\delta_0(x)>\lambda\}) = c_0$ and $\lim\limits_{r\to \infty}\frac{H(r-y)}{H(r)}=1$ for all $y\geq 0${\rm;}
			\item 		
			$\lim\limits_{\nu \to 0^+}\lambda \mu(\{x\geq 0: M_{\mu}\nu(x)>\lambda\}) = c_0\nu(X)$ for all finite measure $\nu$.
		\end{enumerate}
	\end{prop}
	Using Proposition \ref{co:R}, the limiting weak-type behaviors for
	maximal operator associated with several special measures are presented immediately.
	\begin{ex}
		Assume $h$ is a  non-negative measurable periodic function on $[0,\infty)$ and  $\mu$ is a measure
		determined by $d\mu=hdm$.  For all finite measure $\nu$, then
		\begin{equation*}
			\lim\limits_{\lambda \to 0^+}\lambda \mu(\{x\geq 0: M_{\mu}\nu(x)>\lambda\}) = \frac{1}{2} \nu(X).
		\end{equation*}
	\end{ex}
	\begin{ex}
		Let $\mu$ be a measure
		determined by $d\mu(x)=\frac{1}{x+1}\cdot \chi_{[0,\infty)}(x)dm(x)$.  For all finite measure $\nu$, then
		\begin{equation*}
			\lim\limits_{\lambda \to 0^+}\lambda \mu(\{x\geq 0: M_{\mu}\nu(x)>\lambda\}) = \nu(X).
		\end{equation*}
	\end{ex}
	In addition, if $c_0 =0$, we go back to Proposition \ref{pr:2.4} and an example of a non-doubling measure is given here.
	\begin{ex}
		Let $\mu$ be a measure
		determined by $d\mu(x)=e^x\cdot \chi_{[0,\infty)}(x)dm(x)$,	where $e$ denotes the natural exponential constant.  For all finite measure $\nu$, then
		\begin{equation*}
			\lim\limits_{\lambda \to 0^+}\lambda \mu(\{x\geq 0: M_{\mu}\nu(x)>\lambda\}) = 0.
		\end{equation*}
	\end{ex}
	The results in the rest of this section might be helpful to show that the criterion of Proposition \ref{co:R}  cannot be reduced to the existence of a limiting weak type behavior for the maximal operator on one of the particular Dirac deltas. We provide an example of the anomalies that arise in the absence of assumption $\lim\limits_{r\to \infty}\frac{H(r-y)}{H(r)}=1$ for all $y\geq 0$. To do this, we construct a function firstly.
	\begin{lem}\label{lem:3.5}
		There is a monotone increasing real continuous  function $G$
		on $[0,+\infty)$, such that $G(x)$ satisfies the following properties{\rm :}
		\begin{enumerate}[\rm(1)]
			\item $\varliminf\limits_{x\to \infty} \left(G(x+1)-G(x)\right) = 0$ and $\varlimsup\limits_{x\to \infty} \left(G(x+1)-G(x)\right) = 1${\rm ;}\label{eq:985}
			\item $\lim\limits_{x\to \infty} \left(G(x)-G(2x)\right) = -1.$\label{eq:211}
		\end{enumerate}
	\end{lem}
	\begin{proof}
		We prove the lemma by constructing an example.
		For any large integer $n>1$, we use the symbol $k_{n}$ to denote the first ordinal satisfying $\sum_{i=0}^{k_{n}}\frac{1}{n+i} > 1$.
		Fixed real number $x_0\geq1$ and a large integer $n_1\geq1$,
		we next construct $G(x)$ when $x_0\leq x\leq2^{k_{n_1}}x_0+2^{k_{n_1}}$.
		Let the points $A_0^{(1)}=(x_0, 0)$ and $B_0^{(1)}=(x_0+1,1)$, we connect the straight lines $A_0^{(1)}B_0^{(1)}$.
		Set $C_0^{(1)}:=A_1=(2x_0,1)$, $B_1^{(1)}=(2x_0+2,2-\frac1n)$ and $C_1^{(1)}=(2x_0+2+\frac{1}{2^{k_{n_1}-1}},2)$.
		For any $1\leq m\leq k_{n_1}-1$, we denote that
		$
		A_m^{(1)}=(2^mx_0,m)$,\;$B_m^{(1)}=(2^mx_0+2^m,m+1-\sum_{j=1}^m\frac{1}{n+j})$, $C_m^{(1)}=(2^mx_0+2^m+\frac{1}{2^{k_{n_1}-m}},m+1),
		$
		$
		A_{k_{n_1}}^{(1)}=(2^{k_{n_1}}x_0,k_{n_1}+1)$ and $B_{k_{n_1}}^{(1)}=(2^{k_{n_1}}x_0+2^{k_{n_1}},{k_{n_1}}).
		$
		We now connect the straight lines $A_m^{(1)}B_m^{(1)}, B_m^{(1)}C_m^{(1)}$ and $C_m^{(1)}A_{m+1}^{(1)}$, then we get that the image of $G(x)$ on $x_0\leq x\leq2^{k_{n_1}}x_0+2^{k_{n_1}}$ is
		\begin{equation*}\mathcal{S}_1:=\cup_{m=0}^{k_{n_1}-1} \{A_m^{(1)}B_m^{(1)},B_m^{(1)}C_m^{(1)}, C_m^{(1)}A_{m+1}^{(1)}\}\cup \{A_{k_{n_1}}^{(1)}B_{k_{n_1}}^{(1)}\}.
		\end{equation*}
		Taking $x_1=2^{k_{n_1}}x_0+2^{k_{n_1}}$ and $n_2=n_1+k_{n_1}$, by the above structure, we obtain that the image of $G(x)$ on $x_1\leq x\leq2^{k_{n_2}}x_1+2^{k_{n_2}}$ is \begin{equation*}\mathcal{S}_2:=\cup_{m=0}^{k_{n_2}-1} \{A_m^{(2)}B_m^{(2)},B_m^{(2)}C_m^{(2)}, C_m^{(2)}A_{m+1}^{(2)}\}\cup \{A_{k_{n_1}}^{(2)}B_{k_{n_1}}^{(2)}\}.
		\end{equation*}
		By induction, the image of $G(x)$ on $[x_0,+\infty)$ is $\cup_{j=1}^\infty\mathcal{S}_j$.
		
		In the following, we show that the above $G(x)$ satisfies the conditions (1) and (2).
		It is easy to see that the condition (1) holds.
		We now prove $G(x)$ satisfies (2).
		We take a large positive number $x\in\mathbb{R}$ and $(x,G(x))\in \mathcal{S}_l$ for some integer $l>1$.
		This implies that there exists an integer $m$ such that the point $(x,G(x))$ belongs to the lines $A_m^{(l)}B_m^{(l)}$ or $B_m^{(l)}C_m^{(l)}$ or $C_m^{(l)}A_{m+1}^{(l)}$.
		We prove the lemma by the following cases.
		\par \textbf{Case 1}: $(x,G(x))\in A_m^{(l)}B_m^{(l)}$.
		\par
		Note that
		\begin{equation*}
			A_m^{(l)}=(2^mx_l,s),\;\;B_m^{(l)}=(2^mx_l+2^m,s+1-\sum_{j=1}^m\frac{1}{n_l+j})
		\end{equation*}
		for some integer $s$.
		We know
		$(2x,G(2x))\in A_{m+1}^{(l)}B_{m+1}^{(l)}$(if $m=k_{n_l}$, by the above structures, (2) follows).
		It is easy to check that
		\begin{equation*}
			|G(2x)-G(x)-1|\leq\frac{1}{n_l+m+1}\rightarrow0,
		\end{equation*}
		as $l\rightarrow\infty$.
		So (2) follows.
		\par \textbf{Case 2}: $(x,G(x))\in B_m^{(l)}C_m^{(l)}$.
		\par Recall that
		\begin{equation*}
			B_m^{(l)}=(2^mx_l+2^m,s-\sum_{j=1}^m\frac{1}{n_l+j}),\;\;C_m^{(l)}=(2^mx_l+2^m+\frac{1}{2^{k_{n_l}-m}},s).
		\end{equation*}
		If $m= k_{n_l}-1$, then $(2x,G(2x))\in B_0^{(l+1)}C_0^{(l+1)}$.
		Therefore
		\begin{equation*}
			|G(2x)-G(x)-1|\leq 1-\sum_{j=1}^{k_{n_l}-1}\frac{1}{n_l+j}\rightarrow0,
		\end{equation*}
		as $l\rightarrow\infty$, then (2) follows.
		If $m<k_{n_l}-1$, then $(2x,G(2x))\in B_{m+1}^{(l)}C_{m+1}^{(l)}$.
		Similar to the above arguments, one has
		\begin{equation*}
			|G(2x)-G(x)-1|\leq \frac{1}{n_l+m}\rightarrow0,
		\end{equation*}
		as $l\rightarrow\infty$, so (2) follows.
		\par \textbf{Case 3}: $(x,G(x))\in C_m^{(l)}A_{m+1}^{(l)}$.
		\par This case is obvious.
		\par In summary, we prove the lemma.
	\end{proof}
	
	Through this lemma, we can prove that there exists an absolutely continuous measure  satisfying $
		\lim_{\lambda\to 0^+}\lambda \mu\{x\in \mathbb{R}^n: M_{\mu}\delta_{y}(x)>\lambda\}=c_0
	$ for one point $y$ but $\lim\limits_{\lambda\to 0^+}\lambda \mu\{x\in \mathbb{R}^n: M_{\mu}\nu(x)>\lambda\}=c_0\nu(X)$ for all finite measure $\nu$.
	\begin{prop}
		There exists an absolutely continuous measure $\mu$ satisfying the following properties{\rm:}
		\begin{enumerate}[\rm(1)]
			\item $\varlimsup_{\lambda \to 0^+}\lambda \mu(\{x\in X: |M_{\mu}\delta_{2}(x)|>\lambda\}) > e^{-1};$
			\item $\varliminf_{\lambda \to 0^+}\lambda \mu(\{x\in X: |M_{\mu}\delta_{2}(x)|>\lambda\}) = e^{-1};$
			\item $\lim_{\lambda \to 0^+}\lambda \mu(\{x\in X: |M_{\mu}\delta_{0}(x)|>\lambda\}) = e^{-1}.$
		\end{enumerate}
		where $e$ denotes the natural exponential constant.
	\end{prop}
	\begin{proof}
		To prove this proposition, let $H(x)=e^{G(x)},\,x\geq0$, where $G(x)$ is given in Lemma \ref{lem:3.5}. Since $G(x)$ is a positive  monotone increasing function, there exists a measure $\mu$ whose distribution function is $H(x)$. We can see that the measure is locally finite and absolutely continuous with respect to the Lebesgue measure. This follows from the properties \eqref{eq:985} and \eqref{eq:211} in Lemma \ref{lem:3.5} and  Proposition \ref{co:R}.
	\end{proof}
	
\section*{acknowledge}
We appreciate Dr. Lu very much for providing the proof of Lemma \ref{lem:3.5} on our manuscript.


\begin{thebibliography}{999}
		
		\bibitem[Al21]{Al21} Jes\'{u}s Mun\'{a}rriz Aldaz, Kissing numbers and the centered maximal operator. J. Geom. Anal. 31 (2021), no. 10, 10194-10214.
		
		 
		
		\bibitem[Fe69]{Fe69}Herbert Federer,
		Geometric measure theory.
		Die Grundlehren der mathematischen Wissenschaften, Band 153 Springer-Verlag New York Inc., New York 1969 xiv+676 pp.
		
		\bibitem[Gu81]{Gu81} Miguel de Guzm\'{a}n, Real Variable Methods in Fourier Analysis. North-Holland Math. Stud. 46, North-Holland, Amsterdam, 1981.
		
		\bibitem[He81]{He01} Juha Heinonen, Lectures on analysis on metric spaces. Universitext. Springer-Verlag, New York, 2001. x+140 pp.
		
		
		\bibitem[HH08]{HH08}Jiaxin Hu, Xueping Huang,  A note on the limiting weak-type behavior for maximal operators. Proc. Amer. Math. Soc. 136 (2008), no. 5, 1599-1607.
		
		\bibitem[HW19]{HW19}Xianming Hou, Huoxiong Wu, On the limiting weak-type behaviors for maximal operators associated with power weighted measure. Canad. Math. Bull. 62 (2019), no. 2, 313-326.
		
		\bibitem[Ja06]{Ja06} Prabhu Janakiraman, Limiting weak type behavior for singular integral and maximal
		operators. Trans. Amer. Math. Soc. 358(2006), no. 5, 1937-1952.

		\bibitem[Ma95]{Ma95} Pertti Mattila, Geometry of sets and measures in Euclidean spaces. Fractals and rectifiability. Cambridge Studies in Advanced Mathematics, 44. Cambridge University Press, Cambridge, 1995. xii+343 pp.
		
		\bibitem[St15]{St15} Krzysztof Stempak Modified Hardy-Littlewood maximal operators on nondoubling metric measure spaces. Ann. Acad. Sci. Fenn. Math. 40 (2015), no. 1, 443–448. 
		
		
	\end{thebibliography}
\end{document}